\documentclass[a4paper]{amsart}

\usepackage[utf8]{inputenc}
\usepackage[english]{babel}

\usepackage[T1]{fontenc}
\usepackage{lmodern}  

\usepackage{hyperref}
\usepackage{my-math-symbol}
\usepackage{my-cleveref}
\usepackage{my-equation-numbering}

\usepackage{comment}
\usepackage{tikz}
\usetikzlibrary{cd}

\newcommand{\GJMS}{P}
\newcommand{\dilation}{\delta}
\newcommand{\ltrans}{l}
\newcommand{\inversion}{\iota}

\newcommand{\Hsymbol}{S_{H}}
\newcommand{\Hpsido}{\Psi_{H}}
\newcommand{\HSobolev}{W_{H}}

\usepackage[non-sorted-cites,initials,alphabetic,nobysame]{amsrefs}

\AtBeginDocument{%
	\def\MR#1{}
}

\title{Analysis of the critical CR GJMS operator}
\author{Yuya Takeuchi}
\address{Department of Mathematics \\ Graduate School of Science \\ Osaka University
	\\ 1-1 Machikaneyama-cho, Toyonaka, Osaka 560-0043, Japan}
\curraddr{Division of Mathematics \\ Faculty of Pure and Applied Sciences \\ University of Tsukuba
	\\ 1-1-1 Tennodai, Tsukuba, Ibaraki 305-8571 Japan}
\email{ytakeuchi@math.tsukuba.ac.jp, yuya.takeuchi.math@gmail.com}



\thanks{This work was supported by JSPS Research Fellowship for Young Scientists
and JSPS KAKENHI Grant Numbers JP16J04653, JP19J00063, and JP21K13792.}

\begin{document}

\begin{abstract}
	The critical CR GJMS operator on a strictly pseudoconvex CR manifold
	is a non-hypoelliptic CR invariant differential operator.
	We prove that,
	under the embeddability assumption,
	it is essentially self-adjoint and has closed range.
	Moreover,
	its spectrum is discrete,
	and the eigenspace corresponding to each non-zero eigenvalue
	is a finite-dimensional subspace of the space of smooth functions.
	As an application,
	we obtain a necessary and sufficient condition
	for the existence of a contact form with zero CR $Q$-curvature.
\end{abstract}

\maketitle

\tableofcontents

\section{Introduction}
\label{section:introduction}

It is one of the most important topics in both conformal and CR geometries
to study invariant differential operators.
Analytic properties of such operators are deeply connected to geometric problems,
such as the Yamabe problem and the constant $Q$-curvature problem.

In conformal geometry,
Graham, Jenne, Mason, and Sparling~\cite{Graham-Jenne-Mason-Sparling1992} have constructed
a family of conformally invariant differential operators,
called GJMS operators.
Let $(N, g)$ be a Riemannian manifold of dimension $n$.
For $k \in \bbN$ and $k \leq n / 2$ if $n$ is even,
the $k$-th GJMS operator $P_{k}$ is a differential operator acting on $C^{\infty}(N)$
such that its principal part coincides with the $k$-th power of the Laplacian,
and it has the following transformation law under the conformal change $\whg = e^{2 \Upsilon} g$:
\begin{equation}
	e^{(n/2 + k) \Upsilon} \whP_{k}
	= P_{k} e^{(n/2 - k) \Upsilon},
\end{equation}
where $\whP_{k}$ is defined in terms of $\whg$.
Analytic properties of $P_{k}$ on closed manifolds are quite simple.
It follows from standard elliptic theory that
$P_{k}$ is essentially self-adjoint and has closed range.
Moreover,
its spectrum is a discrete subset of $\bbR$,
and the eigenspace corresponding to each eigenvalue is a finite-dimensional subspace of $C^{\infty}(N)$.

In CR geometry,
Gover and Graham~\cite{Gover-Graham2005}
have introduced a family of CR invariant differential operators,
called CR GJMS operators,
via Fefferman construction.
Let $(M, T^{1,0}M, \theta)$ be a $(2n+1)$-dimensional pseudo-Hermitian manifold
and $k \in \bbN$ with $k \leq n+1$.
The \emph{$k$-th CR GJMS operator} $P_{k}$ is a differential operator acting on $C^{\infty}(M)$
such that its principal part is the $k$-th power of the sub-Laplacian,
and its transformation rule under the conformal change $\hat{\theta} = e^{\Upsilon} \theta$
is given by
\begin{equation}
	e^{(n+1+k) \Upsilon / 2} \whP_{k}
	= P_{k} e^{(n+1-k) \Upsilon / 2},
\end{equation}
where $\whP_{k}$ is defined in terms of $\hat{\theta}$.
Although $P_{k}$ is not elliptic,
it is known to be \emph{subelliptic} for $1 \leq k \leq n$~\cite{Ponge2008-Book};
in particular, the same statements as in the previous paragraph
also hold for $P_{k}$ on closed manifolds.
However,
the kernel of the \emph{critical CR GJMS operator} $P_{n + 1}$
contains the space of CR pluriharmonic functions,
which is infinite-dimensional on closed embeddable CR manifolds (\cref{rem:infine-dimensional-space}).
Moreover,
there exist $L^{2}$ non-smooth CR pluriharmonic functions,
which implies that $P_{n + 1}$ is not even hypoelliptic (\cref{rem:L^2-non-smooth}).
In this paper,
nevertheless,
we will prove that similar results to the above are true for $P_{n+1}$
on the orthogonal complement of $\Ker P_{n+1}$.
In what follows,
we simply write $\GJMS$ for the critical CR GJMS operator.

In the remainder of this section,
let $(M, T^{1,0}M, \theta)$ be a closed embeddable pseudo-Hermitian manifold of dimension $2n+1$.
Here,
``embeddable'' means that $(M, T^{1, 0} M)$ can be CR embedded into some $\bbC^{N}$.
Note that the embeddability automatically holds if $n \geq 2$~\cite{Boutet_de_Monvel1975}.
We consider $\GJMS$ as an unbounded operator on $L^{2}(M)$ with domain
\begin{equation}
	\Dom \GJMS
	= \Set{u \in L^{2}(M) | \text{$\GJMS u$ in the weak sense is in $L^{2}(M)$}}.
\end{equation}
We will first prove

\begin{theorem}
\label{thm:self-adjointness-of-critical-CR-GJMS-operator}
	The operator $\GJMS$ is self-adjoint and has closed range.
\end{theorem}

Moreover,
we obtain the following theorem on the spectrum of $\GJMS$:

\begin{theorem}
\label{thm:spectrum-of-critical-CR-GJMS-operator}
	The spectrum of $\GJMS$
	is a discrete subset in $\bbR$
	and consists only of eigenvalues.
	Moreover,
	the eigenspace corresponding to each non-zero eigenvalue of $\GJMS$
	is a finite-dimensional subspace of $C^{\infty}(M)$.
	Furthermore,
	$\Ker \GJMS \cap C^{\infty}(M)$ is dense in $\Ker \GJMS$.
\end{theorem}

In dimension three,
Hsiao~\cite{Hsiao2015} has shown
\cref{thm:self-adjointness-of-critical-CR-GJMS-operator,thm:spectrum-of-critical-CR-GJMS-operator}
by using Fourier integral operators with complex phase.
Our proofs are similar to Hsiao's ones,
but based on the \emph{Heisenberg calculus},
the theory of Heisenberg pseudodifferential operators.
The use of these operators
simplifies some proofs and gives better regularity results.

We will also give some applications
of these theorems and their proofs.
Let $\scrP$ and $\overline{\scrP}$
be the space of CR pluriharmonic functions and its $L^{2}$-closure respectively.
Then $\Ker \GJMS$ contains $\overline{\scrP}$,
and the \emph{supplementary space} $\scrW$ is defined by
\begin{equation}
	\scrW
	\coloneqq \Ker \GJMS \cap \overline{\scrP}^{\perp}.
\end{equation}

\begin{proposition}
\label{prop:supplementary-space}
	The supplementary space $\scrW$
	is a finite dimensional subspace of $C^{\infty}(M)$.
\end{proposition}

In dimension three,
\cref{prop:supplementary-space} has been already proved by Hsiao~\cite{Hsiao2015}.
However,
in this case,
the author~\cite{Takeuchi2020-Paneitz} has shown that $\scrW$ is equal to zero.
On the other hand,
for each $n \geq 2$,
there exists a closed pseudo-Hermitian manifold $(M, T^{1, 0} M, \theta)$ of dimension $2n+1$
such that $\scrW \neq 0$;
see the proof of~\cite{Takeuchi2018}*{Theorem 1.6}.

We will also tackle the zero CR $Q$-curvature problem.
The \emph{CR $Q$-curvature} $Q$,
introduced by Fefferman and Hirachi~\cite{Fefferman-Hirachi2003},
is a smooth function on $M$ such that
it transforms as follows under the conformal change $\whxth = e^{\Upsilon} \theta$:
\begin{equation}
\label{eq:transformation-law-of-CR-Q-curvature}
	\whQ = e^{-(n+1) \Upsilon} (Q + \GJMS \Upsilon),
\end{equation}
where $\whQ$ is defined in terms of $\whxth$.
Marugame~\cite{Marugame2018} has proved that
the total CR $Q$-curvature
\begin{equation}
	\overline{Q}
	\coloneqq \int_{M} Q \, \theta \wedge (d \theta)^{n}
\end{equation}
is always equal to zero.
Moreover,
the CR $Q$-curvature itself is identically zero for pseudo-Einstein contact forms~\cite{Fefferman-Hirachi2003}.
Hence it is natural to ask
whether $(M, T^{1,0}M)$ admits a contact form whose CR $Q$-curvature vanishes identically;
this is the \emph{zero CR $Q$-curvature problem}.
This problem has been solved affirmatively
for embeddable CR three-manifolds by the author~\cite{Takeuchi2020-Paneitz}.
However,
it is still open in general.
By the transformation law \cref{eq:transformation-law-of-CR-Q-curvature},
it is necessary that
\begin{equation}
	\int_{M} f Q \, \theta \wedge (d \theta)^{n} = 0
\end{equation}
holds for any $f \in \Ker \GJMS \cap C^{\infty}(M)$.
Note that this condition is independent of the choice of $\theta$.
The following proposition states that it is also a sufficient condition
for embeddable CR manifolds:

\begin{proposition}
\label{prop:zero-CR-Q-curvature-problem}
	There exists a contact form $\hat{\theta}$ on $M$
	such that the CR $Q$-curvature $\whQ$ vanishes identically
	if and only if $Q \perp (\Ker \GJMS \cap C^{\infty}(M))$.
\end{proposition}

This paper is organized as follows.
In \cref{section:CR-manifolds},
we recall basic facts on CR manifolds.
\cref{section:model-operators-on-the-Heisenberg-group}
deals with convolution operators on the Heisenberg group,
which is a ``model'' of the Heisenberg calculus.
In \cref{section:Heisenberg-calculus},
we give a brief exposition of the Heisenberg calculus.
\cref{section:proofs-of-the-main-results}
is devoted to proofs of the main results in this paper.

\section{CR manifolds}
\label{section:CR-manifolds}

Let $M$ be an orientable smooth $(2n+1)$-dimensional manifold without boundary.
A \emph{CR structure} is a rank $n$ complex subbundle $T^{1, 0} M$
of the complexified tangent bundle $T M \otimes \mathbb{C}$ such that
\begin{equation}
	T^{1, 0} M \cap T^{0, 1} M = 0, \qquad
	\comm{\Gamma(T^{1 ,0} M)}{\Gamma(T^{1, 0} M)} \subset \Gamma(T^{1, 0} M),
\end{equation}
where $T^{0, 1} M$ is the complex conjugate of $T^{1, 0} M$ in $T M \otimes \bbC$.
Define a hyperplane bundle $H M$ of $T M$ by $H M \coloneqq \Re T^{1, 0} M$.
A typical example of CR manifolds is a real hypersurface $M$ in an $(n+1)$-dimensional complex manifold $X$;
this $M$ has the canonical CR structure
\begin{equation}
	T^{1, 0} M
	\coloneqq T^{1, 0} X |_{M} \cap (T M \otimes \bbC).
\end{equation}

Take a nowhere-vanishing real one-form $\theta$ on $M$
such that $\theta$ annihilates $T^{1, 0} M$.
The \emph{Levi form} $\mathcal{L}_{\theta}$ with respect to $\theta$ is the Hermitian form
on $T^{1,0} M$ defined by
\begin{equation}
	\calL_{\theta}(Z, W)
	\coloneqq - \sqrt{-1} \, d \theta(Z, \ovW), \qquad Z, W \in T^{1, 0} M.
\end{equation}
A CR structure $T^{1, 0} M$ is said to be \emph{strictly pseudoconvex}
if the Levi form is positive definite for some $\theta$;
such a $\theta$ is called a \emph{contact form}.
The triple $(M, T^{1, 0} M, \theta)$ is called a \emph{pseudo-Hermitian manifold}.
Denote by $T$ the \emph{Reeb vector field} with respect to $\theta$; 
that is,
the unique vector field satisfying
\begin{equation}
	\theta(T) = 1, \qquad T \contr d\theta = 0.
\end{equation}

Define an operator $\delb_{b} \colon C^{\infty}(M) \to \Gamma((T^{0, 1} M)^{\ast})$ by
\begin{equation}
	\delb_{b} f
	\coloneqq d f|_{T^{0, 1} M}.
\end{equation}
A smooth function $f$ is called a \emph{CR holomorphic function}
if $\delb_{b} f = 0$.
A \emph{CR pluriharmonic function} is a real-valued smooth function
that is locally the real part of a CR holomorphic function.
We denote by $\scrP$ the space of CR pluriharmonic functions.

\begin{remark}
\label{rem:infine-dimensional-space}
	It is known that the spaces of CR holomorphic functions and CR pluriharmonic functions
	are infinite-dimensional if there exists a not locally constant CR holomorphic function $f$.
	Suppose to the contrary that the space of CR holomorphic functions is finite-dimensional.
	Then $f$ is algebraically dependent over $\bbC$
	since $f^{k}$ is also CR holomorphic for $k \in \bbN$.
	This implies that $f$ is locally constant,
	which is a contradiction.
	Taking the real part yields that
	the space of CR pluriharmonic functions is also infinite-dimensional.
\end{remark}

\begin{remark}
\label{rem:L^2-non-smooth}
	Let
	\begin{equation}
		S^{2 n + 1}
		= \Set{z = (z^{1}, \dots , z^{n + 1}) \in \bbC^{n + 1} | \abs{z}^{2} = 1}
	\end{equation}
	be the unit sphere in $\bbC^{n + 1}$ with the standard CR structure.
	The function $u = \log \abs{1 - z^{1}}^{2}$ is $L^{2}$ but not continuous.
	Moreover,
	$u_{\varepsilon} = \log \abs{1 + \varepsilon - z^{1}}^{2}$ is CR pluriharmonic
	and $u_{\varepsilon} \to u$ as $\varepsilon \to + 0$ in $L^{2}(S^{2 n + 1})$.
	Hence $u$ is an example of $L^{2}$ non-smooth CR pluriharmonic functions.
\end{remark}

The Levi form induces a Hermitian metric on $(T^{0, 1} M)^{\ast}$.
By using this Hermitian metric and the volume form $\theta \wedge (d \theta)^{n}$,
we obtain the formal adjoint
$\delb_{b}^{\ast} \colon \Gamma((T^{0, 1} M)^{\ast}) \to C^{\infty}(M)$ of $\delb_{b}$.
The \emph{Kohn Laplacian} $\Box_{b}$ and the \emph{sub-Laplacian} $\Delta_{b}$
are defined by
\begin{equation}
	\Box_{b}
	\coloneqq \delb_{b}^{\ast} \delb_{b},
	\qquad
	\Delta_{b}
	\coloneqq \Box_{b} + \overline{\Box}_{b}.
\end{equation}
Note that
\begin{equation}
	\Box_{b}
	= \frac{1}{2} \Delta_{b} + \frac{\sqrt{-1}}{2} n T;
\end{equation}
see~\cite{Lee1986}*{Theorem  2.3} for example.
The Gaffney extension of the Kohn Laplacian,
also denoted by $\Box_{b}$,
is a self-adjoint operator on $L^{2}(M)$.
The kernel $\Ker \Box_{b}$ is the space of $L^{2}$ CR holomorphic functions.

The \emph{critical CR GJMS operator} $\GJMS$
is a real differential operator of order $2n+2$ acting on $C^{\infty}(M)$.
It is known to be formally self-adjoint~\cite{Gover-Graham2005}*{Proposition 5.1}.
Moreover,
it annihilates CR pluriharmonic functions~\cite{Hirachi2014}*{Section 3.2}.

A CR manifold $(M, T^{1,0}M)$ is said to be \emph{embeddable}
if there exists a smooth embedding of $M$ to some $\mathbb{C}^{N}$
such that $T^{1, 0} M = T^{1,0} \mathbb{C}^{N}|_{M} \cap (T M \otimes \mathbb{C})$.
It is known that a closed strictly pseudoconvex CR manifold $(M, T^{1,0}M)$ is embeddable
if and only if $\Box_{b}$ has closed range~\cites{Boutet_de_Monvel1975,Kohn1986}.

\section{Model operators on the Heisenberg group}
\label{section:model-operators-on-the-Heisenberg-group}

The Heisenberg group $G$ is the Lie group with the underlying manifold $\bbR \times \bbC^{n}$
and the multiplication
\begin{equation}
	(t, z) \cdot (t^{\prime}, z^{\prime})
	\coloneqq (t + t^{\prime} + 2 \Im (z \cdot \ovz^{\prime}), z + z^{\prime}).
\end{equation}
The left translation by $(t, z)$ and the inversion on $G$
are denoted by $\ltrans_{(t, z)}$ and $\inversion$ respectively.

For $\alpha = 1, \dots , n$,
we introduce a left-invariant complex vector field $Z_{\alpha}^{0}$ by
\begin{equation}
	Z_{\alpha}^{0}
	\coloneqq \pdv{}{z^{\alpha}} + \sqrt{-1} \ovz^{\alpha} \pdv{}{t}.
\end{equation}
The canonical CR structure $T^{1, 0} G$
is spanned by $Z_{1}^{0}, \dots , Z_{n}^{0}$.
Define a left-invariant one-form $\theta^{0}$ on $G$ by
\begin{equation}
	\theta^{0}
	\coloneqq d t + \sqrt{-1} \sum_{\alpha = 1}^{n}
		(z^{\alpha} d \ovz^{\alpha} - \ovz^{\alpha} d z^{\alpha}).
\end{equation}
Then $\theta^{0}$ annihilates $T^{1, 0} G$
and the Levi form $\calL_{\theta^{0}}$ satisfies
$\calL_{\theta^{0}}(Z_{\alpha}^{0}, Z_{\beta}^{0}) = 2 \delta_{\alpha \beta}$;
in particular,
$\theta^{0}$ is a contact form on $G$.
The Reeb vector field $T^{0}$ coincides with $\pdvf{}{t}$.

The Lie algebra $\frakg$ of $G$ is isomorphic to $\bbR \times \bbC^{n}$ as a linear space via
\begin{equation}
	\frakg \to \bbR \times \bbC^{n};
	\qquad
	t T^{0} + 2 \sum_{\alpha = 1}^{n} \Re(z^{\alpha} Z_{\alpha}^{0})
	\mapsto (t, z).
\end{equation}
Under this identification,
the Lie bracket on $\frakg$ is given by
\begin{equation}
	\comm{(t, z)}{(t^{\prime}, z^{\prime})}
	= (4 \Im(z \cdot \ovz^{\prime}), 0).
\end{equation}
Moreover,
the exponential map $\frakg \to G$ coincides with the identity map on $\bbR \times \bbC^{n}$.
Furthermore,
the dual $\frakg^{\ast}$ of $\frakg$ is also canonically isomorphic to $\bbR \times \bbC^{n}$
as a linear space.
We write this linear coordinate as $(\tau, \zeta)$.

For $r \in \bbR_{+}$,
the parabolic dilation $\dilation_{r}$ on $\bbR \times \bbC^{n}$ is defined by
\begin{equation}
	\dilation_{r}(t, z)
	= (r^{2} t, r z).
\end{equation}
This dilation defines automorphisms on $G$, $\frakg$, and $\frakg^{\ast}$,
for which we will use the same letter $\dilation_{r}$ by abuse of notation.
In what follows,
the term ``homogeneous'' is defined in terms of $\dilation_{r}$.
We will sometimes write $v$ for a point of $G$.
Denote by $d v$ the Lebesgue measure on $G$,
which is a Haar measure on $G$.

Let $\Schwartz(G)$ (resp.\ $\Schwartz(\frakg^{\ast})$)
be the space of rapidly decreasing functions on $G$ (resp.\ $\frakg^{\ast}$),
and $\Schwartz^{\prime}(G)$ (resp.\ $\Schwartz^{\prime}(\frakg^{\ast})$)
be that of tempered distributions on $G$ (resp.\ $\frakg^{\ast}$).
The coupling of $f \in \Schwartz(G)$ and $k \in \Schwartz^{\prime}(G)$ is written as $\coupling{k}{f}$.
The pull-back by $\dilation_{r}$
induces endomorphisms on $\Schwartz(G)$ and $\Schwartz(\frakg^{\ast})$,
and these extend to those on $\Schwartz^{\prime}(G)$ and $\Schwartz^{\prime}(\frakg^{\ast})$.
The Fourier transform $\Fourier$ defines isomorphisms
\begin{equation}
	\Schwartz(G) \xrightarrow{\cong} \Schwartz(\frakg^{\ast}),
	\qquad
	\Schwartz^{\prime}(G) \xrightarrow{\cong} \Schwartz^{\prime}(\frakg^{\ast});
\end{equation}
in our convention,
the Fourier transform $\Fourier (f)$ of $f \in \Schwartz(G)$ is defined by
\begin{equation}
	\Fourier (f)(\tau, \zeta)
	\coloneqq \int_{G} e^{- \sqrt{-1} (t \tau + \Re(z \cdot \ovxz))} f(t, z) d v.
\end{equation}

Now we consider ``model operators'' of the Heisenberg calculus.
For $m \in \bbR$,
set
\begin{equation}
	\Hsymbol^{m}
	\coloneqq
	\Set{a \in C^{\infty}(\frakg^{\ast} \setminus \{0\}) | \dilation_{r}^{\ast} a = r^{m} a},
\end{equation}
which is the space of \emph{Heisenberg symbols} of order $m$.
Let $\scrG^{m}$ be the space of $g \in \Schwartz^{\prime}(\frakg^{\ast})$ such that
$g$ is smooth on $\frakg^{\ast} \setminus \{0\}$ and satisfies
\begin{equation}
	\dilation_{r}^{\ast} g
	= r^{m} g + (r^{m} \log r) h,
\end{equation}
where $h \in \Schwartz^{\prime}(\frakg^{\ast})$
with $\supp h \subset \{0\}$ and $\dilation_{r}^{\ast} h = r^{m} h$.
The restriction map $\scrG^{m} \to \Hsymbol^{m}$ is known to be surjective~\cite{Beals-Greiner1988}*{Proposition 15.8}.
Moreover,
the inverse Fourier transform gives an isomorphism
\begin{equation}
	\Fourier^{-1} \colon \scrG^{m} \xrightarrow{\cong} \scrK_{-m-2n-2},
\end{equation}
where $\scrK_{l}$ is the space of $k \in \Schwartz^{\prime}(G)$ such that
$k$ is smooth on $G \setminus \{0\}$ and satisfies
\begin{equation}
	\dilation_{r}^{\ast} k
	= r^{l} k + (r^{l} \log r) \psi
\end{equation}
for a homogeneous polynomial $\psi$ of degree $l$~\cite{Beals-Greiner1988}*{Proposition 15.24}.
We also introduce a function space on which Heisenberg symbols act.
Let $\Schwartz_{0}(G)$ be the space of $f \in \Schwartz(G)$ such that
\begin{equation}
	\int_{G} \psi(v) f(v) d v = 0
\end{equation}
for any polynomial $\psi$ on $G$.
This condition is equivalent to that $\Fourier (f) \in \Schwartz(\frakg^{\ast})$
vanishes to infinite order at the origin.

We denote by $\Hpsido^{m}$ the space of endomorphisms $A$ on $\Schwartz_{0}(G)$
commuting with left translation and admitting its formal adjoint $A^{\ast}$ of homogeneous degree $m$;
that is,
\begin{equation}
	A^{\ast} \circ \dilation_{r}^{\ast}
	= r^{m} \dilation_{r}^{\ast} \circ A^{\ast}.
\end{equation}
We would like to define a canonical isomorphism between $\Hsymbol^{m}$ and $\Hpsido^{m}$.

\begin{proposition}
\label{prop:well-defined-of-Hpsido}
	Let $a \in \Hsymbol^{m}$ and
	take $g \in \scrG^{m}$ with $g |_{\frakg^{\ast} \setminus \{0\}} = a$.
	Then the convolution operator
	\begin{equation}
	\label{eq:definition-of-O(p)}
		f \mapsto [\Fourier^{-1}(g) \ast f](v) \coloneqq \coupling{\Fourier^{-1}(g)}{f \circ \ltrans_{v} \circ \inversion}
	\end{equation}
	defines an endomorphism on $\Schwartz_{0}(G)$
	and is independent of the choice of $g$.
	Moreover,
	this operator commutes with left translation and is homogeneous of degree $m$.
	Furthermore,
	it is equal to zero if and only if $a = 0$.
\end{proposition}

\begin{definition}
	For $a \in \Hsymbol^{m}$,
	an operator $O^{0}(a) \colon \Schwartz_{0}(G) \to \Schwartz_{0}(G)$
	is defined by \cref{eq:definition-of-O(p)}.
\end{definition}

\begin{proof}[Proof of \cref{prop:well-defined-of-Hpsido}]
	It follows from~\cite{Christ-Geller-Glowacki-Polin1992}*{Proposition 2.2}
	that \cref{eq:definition-of-O(p)} defines an endomorphism on $\Schwartz_{0}(G)$
	commuting with left translation and homogeneous of degree $m$.
	Assume that $g^{\prime}$ also satisfies $g^{\prime} |_{\frakg^{\ast} \setminus \{0\}} = a$.
	Then the support of $g^{\prime} - g$ is contained in $\{0\} \subset \frakg^{\ast}$.
	Hence $\Fourier^{-1}(g^{\prime} - g)$ is a polynomial on $G$,
	and so $\Fourier^{-1}(g^{\prime} - g) \ast f = 0$ for any $f \in \Schwartz_{0}(G)$.
	This implies the independence of the choice of $g$.
	Next,
	suppose that the operator \cref{eq:definition-of-O(p)} is equal to zero.
	For any $f \in \Schwartz_{0}(G)$,
	we have $\coupling{\Fourier^{-1}(g)}{f \circ \inversion} = 0$.
	Hence $g$ annihilates $\Fourier(\Schwartz_{0}(G))$.
	Since $C^{\infty}_{c}(\frakg^{\ast} \setminus \{0\})$ is a subspace of $\Fourier(\Schwartz_{0}(G))$,
	the support of $g$ is contained in $\{0\} \subset \frakg^{\ast}$.
	Therefore $a = g |_{\frakg^{\ast} \setminus \{0\}} = 0$.
\end{proof}

The operator $O^{0}(a)$ is well-behaved under formal adjoint and composition.

\begin{theorem}
\label{thm:adjoint-and-composition-of-model-operators}
	(i) The formal adjoint of $O^{0}(a)$,
	$a \in \Hsymbol^{m}$,
	is given by $O^{0}(\ova)$.
	In particular,
	$O^{0}(a)$ is formally self-adjoint if and only if $a$ is real-valued.

	(ii) There exists a bilinear product
	\begin{equation}
		\ast^{0} \colon \Hsymbol^{m_{1}} \times \Hsymbol^{m_{2}} \to \Hsymbol^{m_{1} + m_{2}}
	\end{equation}
	such that $O^{0}(a_{1}) O^{0}(a_{2}) = O^{0}(a_{1} \ast^{0} a_{2})$
	for any $a_{1} \in \Hsymbol^{m_{1}}$ and $a_{2} \in \Hsymbol^{m_{2}}$.
\end{theorem}

\begin{proof}
	(i) Take $g \in \scrG^{m}$ with $g |_{\frakg^{\ast} \setminus \{0\}} = a$
	The formal adjoint of $O^{0}(a)$ is given by
	the convolution with respect to 
	\begin{equation}
		\overline{\Fourier^{-1}(g)} \circ \inversion
		= \Fourier^{-1}(\ovg);
	\end{equation}
	see~\cite{Christ-Geller-Glowacki-Polin1992}*{Section 3}.
	Thus we have $(O^{0}(a))^{\ast} = O^{0}(\ova)$.

	(ii) See \cite{Ponge2008-Book}*{Proposition 3.1.3(2)}.
\end{proof}

In particular,
$O^{0}$ defines an injective map from $\Hsymbol^{m}$ to $\Hpsido^{m}$.
In fact,
this is an isomorphism.

\begin{proposition}
	For any $A \in \Hpsido^{m}$,
	there exists the unique $a \in \Hsymbol^{m}$ such that $A = O^{0}(a)$.
\end{proposition}

\begin{proof}
	Let $A \in \Hpsido^{m}$.
	By~\cite{Christ-Geller-Glowacki-Polin1992}*{Proposition 3.2},
	we have $k \in \scrK_{-m-2n-2}$ such that $A f = k \ast f$ for any $f \in \Schwartz_{0}(G)$.
	If we define $a \in \Hsymbol^{m}$ by $a \coloneqq \Fourier (k) |_{\frakg^{\ast} \setminus \{0\}}$,
	then $O^{0}(a)$ coincides with $A$ by definition.
\end{proof}

\begin{definition}
	The \emph{Heisenberg symbol}
	\begin{equation}
		\sigma_{m}^{0} \colon \Hpsido^{m} \to \Hsymbol^{m}
	\end{equation}
	is defined by the inverse map of $O^{0}$.
\end{definition}

It follows from \cref{thm:adjoint-and-composition-of-model-operators} that
\begin{equation}
	\sigma_{m}^{0}(A^{\ast})
	= \overline{\sigma_{m}^{0}(A)},
	\qquad
	\sigma_{m_{1} + m_{2}}^{0}(A_{1} A_{2})
	= \sigma_{m_{1}}^{0}(A_{1}) \ast^{0} \sigma_{m_{2}}^{0}(A_{2})
\end{equation}
for $A \in \Hpsido^{m}$, $A_{1} \in \Hpsido^{m_{1}}$, and $A_{2} \in \Hpsido^{m_{2}}$.
In particular,
$A$ is formally self-adjoint if and only if $\sigma_{m}^{0}(A)$ is real-valued.

Before the end of this section,
we note a relation between the Reeb vector field and $\Hpsido^{m}$.

\begin{lemma}
\label{lem:Reeb-commutes-model-operator}
	The Reeb vector field $T^{0}$ commutes with any $A \in \Hpsido^{m}$.
\end{lemma}

\begin{proof}
	The vector field $T^{0}$ generates the flow $\ltrans_{(t, 0)}$.
	Since $A \in \Hpsido^{m}$ commutes with left translation,
	we have $\comm{T^{0}}{A} = 0$.
\end{proof}

\section{Heisenberg calculus}
\label{section:Heisenberg-calculus}

In this section,
we recall basic properties of Heisenberg pseudodifferential operators;
see~\cites{Beals-Greiner1988,Ponge2008-Book}
for a comprehensive introduction to the Heisenberg calculus.

Throughout this section,
we fix a closed pseudo-Hermitian manifold $(M, T^{1,0}M, \theta)$ of dimension $2n+1$.
Let
\begin{equation}
	\frakg M
	\coloneqq (TM / HM) \oplus HM.
\end{equation}
The Reeb vector field $T$ defines a nowhere-vanishing section $[T]$ of $TM / HM$.
For sections $X_{0}$ and $Y_{0}$ of $TM / HM$
and $X^{\prime}$ and $Y^{\prime}$ of $HM$,
the Lie bracket $\comm{X_{0} + X^{\prime}}{Y_{0} + Y^{\prime}}$ is defined by
\begin{equation}
	\comm{X_{0} + X^{\prime}}{Y_{0} + Y^{\prime}}
	\coloneqq - d \theta (X^{\prime}, Y^{\prime}) [T].
\end{equation}
This bracket makes $\frakg M$ a bundle of two-step nilpotent Lie algebras.
The dilation $\dilation_{r}$ on $\frakg M$ is defined by
\begin{equation}
	\dilation_{r} |_{TM / HM}
	\coloneqq r^{2},
	\qquad
	\dilation_{r} |_{HM}
	\coloneqq r.
\end{equation}
It follow from the definition of the Lie bracket that
$\dilation_{r}$ is a fiberwise Lie algebra isomorphism.
Set $G M \coloneqq \frakg M$ as a smooth fiber bundle
with the fiberwise group structure defined via the Baker-Campbell-Hausdorff formula.
The dilation $\dilation_{r}$ on $\frakg M$ induces that on $G M$,
which we write as $\dilation_{r}$ for abbreviation.

Take a local frame $(Z_{\alpha})$ of $T^{1,0}M$ on an open set $U \subset M$
such that
\begin{equation}
	\calL_{\theta}(Z_{\alpha}, Z_{\beta}) = 2 \tensor{\delta}{_{\alpha}_{\beta}}.
\end{equation}
Then the map
\begin{equation}
\label{eq:identification-with-trivial-Lie-alg-bundle}
	\frakg M |_{U} \to U \times \frakg;
	\qquad \pqty{ p, t [T] + 2 \Re \sum_{\alpha = 1}^{n} z^{\alpha} Z_{\alpha} }
		\mapsto (p, t, z)
\end{equation}
gives an isomorphism between fiber bundles of Lie algebras.
This isomorphism is compatible with the dilation.
The identification \cref{eq:identification-with-trivial-Lie-alg-bundle} induces those on $G M$
and the dual bundle $\frakg^{\ast} M \coloneqq (\frakg M)^{\ast}$ of $\frakg M$:
\begin{equation}
\label{eq:identification-with-trivial-Lie-grp-bundle}
	G M |_{U} \to U \times G,
	\qquad
	\frakg^{\ast} M |_{U} \to U \times \frakg^{\ast}.
\end{equation}
These are also compatible with the dilation.
Let $(Z_{\alpha}^{\prime})$ be another local frame of $T^{1,0}M$ on $U$ satisfying
$\calL_{\theta}(Z_{\alpha}^{\prime}, Z_{\beta}^{\prime}) = 2 \tensor{\delta}{_{\alpha}_{\beta}}$.
This gives another identification $\frakg M |_{U} \to U \times \frakg$.
These two identifications relate with each other
by a smooth family $(U(p))_{p \in U}$ of unitary matrices;
that is,
\begin{equation}
	U \times \frakg \to U \times \frakg;
	\qquad
	(p, t, z)
	\mapsto (p, t, U(p) \cdot z).
\end{equation}
The same is true for $G M$ and $\frakg^{\ast} M$.

For $m \in \mathbb{R}$,
the space $\Hsymbol^{m}(M)$
consists of functions in $C^{\infty}(\frakg^{*} M \setminus \{0\})$
that are homogeneous of degree $m$ on each fiber.
Under the identification \cref{eq:identification-with-trivial-Lie-grp-bundle},
the fiberwise product $\ast^{0}$ induces a well-defined bilinear product
\begin{equation}
	\ast \colon \Hsymbol^{m_{1}}(M) \times \Hsymbol^{m_{2}}(M)
	\to \Hsymbol^{m_{1} + m_{2}}(M).
\end{equation}

Now we consider Heisenberg pseudodifferential operators.
For $m \in \mathbb{R}$,
denote by $\Hpsido^{m}(M)$
the space of \emph{Heisenberg pseudodifferential operators
$A \colon C^{\infty}(M) \to C^{\infty}(M)$ of order $m$}.
This space is closed under complex conjugate, transpose, and formal adjoint~\cite{Ponge2008-Book}*{Proposition 3.1.23}.
In particular,
any $A \in \Hpsido^{m}$ extends to a linear operator
\begin{equation}
	A \colon \scrD^{\prime}(M) \to \scrD^{\prime}(M),
\end{equation}
where $\scrD^{\prime}(M)$ is the space of distributions on $M$.
For example,
$V \in \Gamma(H M)$ is an element of $\Hpsido^{1}(M)$
and $T \in \Hpsido^{2}(M)$.
Note that $\Hpsido^{-\infty}(M) \coloneqq \bigcap_{m \in \bbZ} \Hpsido^{m}(M)$
coincides with the space of smoothing operators on $M$.
As in the usual pseudodifferential calculus,
there exists the \emph{Heisenberg principal symbol}
\begin{equation}
	\sigma_{m} \colon \Hpsido^{m}(M) \to \Hsymbol^{m}(M),
\end{equation}
which has the following properties:

\begin{proposition}[\cite{Ponge2008-Book}*{Propositions 3.2.6 and 3.2.9}]
\label{prop:Heisenberg-principal-symbol}
	(i) The Heisenberg principal symbol $\sigma_{m}$ gives the following exact sequence:
	\begin{equation}
		0 \to \Hpsido^{m-1}(M) \to \Hpsido^{m}(M) \xrightarrow{\sigma_{m}} \Hsymbol^{m}(M) \to 0.
	\end{equation}

	(ii) For $A_{1} \in \Hpsido^{m_{1}}(M)$ and $A_{2} \in \Hpsido^{m_{2}}(M)$,
	the operator $A_{1} A_{2}$ is a Heisenberg pseudodifferential operator of order $m_{1} + m_{2}$,
	and
	\begin{equation}
		\sigma_{m_{1} + m_{2}}(A_{1} A_{2}) = \sigma_{m_{1}}(A_{1}) \ast \sigma_{m_{2}}(A_{2}).
	\end{equation}
\end{proposition}

On the other hand,
there exists a crucial difference between the usual pseudodifferential calculus and the Heisenberg one.
Since the product $\ast$ is non-commutative,
the commutator $\comm{A_{1}}{A_{2}}$ of $A_{1} \in \Hpsido^{m_{1}}(M)$ and $A_{2} \in \Hpsido^{m_{2}}(M)$
is not an element of $\Hpsido^{m_{1} + m_{2} - 1}(M)$ in general.
However,
we have the following

\begin{lemma}
\label{lem:commutator-for-Reeb-vector-field}
	Let $A \in \Hpsido^{m}(M)$.
	Then $\comm{T}{A} \in \Hpsido^{m+1}(M)$.
\end{lemma}

\begin{proof}
	It is enough to show that $\sigma_{m+2}(\comm{T}{A}) = 0$,
	or equivalently,
	\begin{equation}
		\sigma_{2}(T) \ast \sigma_{m}(A) = \sigma_{m}(A) \ast \sigma_{2}(T).
	\end{equation}
	Fix an identification \cref{eq:identification-with-trivial-Lie-grp-bundle}.
	Then $\sigma_{2}(T) \in \Hsymbol^{2}(M)$ is given by
	\begin{equation}
		\sigma_{2}(T)(p, \tau, \zeta)
		= \sqrt{-1} \tau
		= \sigma_{2}^{0}(T^{0})(\tau, \zeta);
	\end{equation}
	see~\cite{Ponge2008-Book}*{Example 3.2.5}.
	Hence it suffices to prove that
	$\sigma_{2}^{0}(T^{0}) \ast^{0} a = a \ast^{0} \sigma_{2}^{0}(T^{0})$
	holds for any $a \in \Hsymbol^{m}$. 
	From \cref{lem:Reeb-commutes-model-operator},
	we obtain
	\begin{equation}
		O^{0}(\sigma_{2}^{0}(T^{0}) \ast^{0} a)
		= T^{0} O^{0}(a)
		= O^{0}(a) T^{0}
		= O^{0}(a \ast^{0} \sigma_{2}^{0}(T^{0})),
	\end{equation}
	which is equivalent to
	$\sigma_{2}^{0}(T^{0}) \ast^{0} a = a \ast^{0} \sigma_{2}^{0}(T^{0})$.
\end{proof}

Next,
consider approximate inverses of Heisenberg pseudodifferential operators.
We write $A \sim B$ if $A - B$ is a smoothing operator.

\begin{definition}
	Let $A \in \Hpsido^{m}(M)$.
	An operator $B \in \Hpsido^{-m}(M)$ is called a \emph{parametrix} of $A$
	if $A B \sim I$ and $B A \sim I$.
\end{definition}

The existence of a parametrix of a Heisenberg pseudodifferential operator
is determined only by its Heisenberg principal symbol.

\begin{proposition}[\cite{Ponge2008-Book}*{Proposition 3.3.1}]
\label{prop:equivalent-conditions-for-existence-of-parametrix}
	Let $A \in \Hpsido^{m}(M)$
	with Heisenberg principal symbol $a \in \Hsymbol^{m}(M)$.
	Then $A$ has a parametrix if and only if
	there exists $b \in \Hsymbol^{-m}(M)$ such that
	$a \ast b = b \ast a = 1$.
\end{proposition}

Now consider the Heisenberg differential operator $\Delta_{b} + 1$ of order $2$.
It is known that this operator has a parametrix;
see the proof of~\cite{Ponge2008-Book}*{Proposition 3.5.7} for example.
Since $\Delta_{b} + 1$ is positive and self-adjoint,
the $s$-th power $(\Delta_{b} + 1)^{s}$ of $\Delta_{b} + 1$,
$s \in \mathbb{R}$,
is a Heisenberg pseudodifferential operator of order $2 s$~\cite{Ponge2008-Book}*{Theorems 5.3.1 and 5.4.10}.
Using this operator,
we define
\begin{equation}
	\HSobolev^{s}(M)
	:= \Set{ u \in \scrD^{\prime}(M) \mid (\Delta_{b} + 1)^{s/2} u \in L^{2}(M) }.
\end{equation}
This space is a Hilbert space with the inner product
\begin{equation}
	\iproduct{u}{v}_{s}
	= \iproduct{(\Delta_{b} + 1)^{s/2} u}{(\Delta_{b} + 1)^{s/2} v}_{L^{2}(M)};
\end{equation}
write $\norm{\cdot}_{s}$ for the norm determined by $\iproduct{\cdot}{\cdot}_{s}$.
The space $C^{\infty}(M)$ is dense in $\HSobolev^{s}(M)$,
and $C^{\infty}(M) = \bigcap_{s \in \bbR} \HSobolev^{s}(M)$~\cite{Ponge2008-Book}*{Proposition 5.5.3}.
Note that,
for $k \in \mathbb{N}$,
the Hilbert space $\HSobolev^{k}(M)$ coincides with the Folland-Stein space $S^{k, 2}(M)$
as a topological vector space~\cite{Ponge2008-Book}*{Proposition 5.5.5}.
Similar to the usual $L^{2}$-Sobolev space theory,
we obtain the following

\begin{lemma}
\label{lem:Rellich's-lemma}
	For $s_{1} < s_{2}$,
	the embedding $\HSobolev^{s_{2}}(M) \hookrightarrow \HSobolev^{s_{1}}(M)$ is compact.
\end{lemma}

\begin{proof}
	The operator $(\Delta_{b} + 1)^{s^{\prime}/2}$, $s^{\prime} \in \mathbb{R}$,
	gives an isometry $\HSobolev^{s + s^{\prime}}(M) \to \HSobolev^{s}(M)$,
	and so we may assume that $s_{1} = 0$.
	From~\cite{Ponge2008-Book}*{Proposition 5.5.7},
	we derive that
	the embedding $\HSobolev^{s_{2}}(M) \hookrightarrow \HSobolev^{0}(M) = L^{2}(M)$
	is the composition of the two embeddings
	$\HSobolev^{s_{2}}(M) \hookrightarrow H^{s_{2} / 2}(M)$ and $H^{s_{2} / 2}(M) \hookrightarrow L^{2}(M)$,
	where $H^{s}(M)$ is the usual $L^{2}$-Sobolev space on $M$ of order $s$.
	Thus the compactness of $\HSobolev^{s_{2}}(M) \hookrightarrow L^{2}(M)$
	follows from Rellich's lemma.
\end{proof}

Heisenberg pseudodifferential operators act on these Hilbert spaces as follows:

\begin{proposition}
\label{prop:mapping-properties-of-Hpsido}
	Any $A \in \Hpsido^{m}(M)$
	extends to a continuous linear operator
	\begin{equation}
		A \colon \HSobolev^{s+m}(M) \to \HSobolev^{s}(M)
	\end{equation}
	for every $s \in \bbR$.
	In particular if $m < 0$,
	the operator $A \colon L^{2}(M) \to L^{2}(M)$ is compact.
\end{proposition}

\begin{proof}
	The former statement follows from~\cite{Ponge2008-Book}*{Propositions 5.5.8}.
	The latter one is a consequence of the former one and \cref{lem:Rellich's-lemma}.
\end{proof}

\section{Proofs of the main results}
\label{section:proofs-of-the-main-results}

In this section,
we prove the main results in this paper.
In what follows,
we fix a closed embeddable pseudo-Hermitian manifold  $(M, T^{1, 0} M, \theta)$ of dimension $2n+1$.

For $\mu \in \bbR$,
we define a formally self-adjoint Heisenberg differential operator $L_{\mu}$ of order $2$ by
\begin{equation}
	L_{\mu}
	\coloneqq \frac{1}{2} \Delta_{b} + \frac{\sqrt{-1}}{2} \mu T.
\end{equation}
It is known that
$L_{\mu}$ has a parametrix $N_{\mu} \in \Hpsido^{-2}(M)$
if and only if $\mu \notin \pm (n + 2 \bbN)$;
see the proof of~\cite{Ponge2008-Book}*{Proposition 3.5.7} for example.
On the other hand,
the embeddability of $M$ implies that
there exist the partial inverse $N_{n} \in \Hpsido^{-2}(M)$ of $L_{n} = \Box_{b}$
and the orthogonal projection $S \in \Hpsido^{0}(M)$ to $\Ker \Box_{b}$,
called the \emph{\Szego projection}~\cite{Beals-Greiner1988}*{Theorem 24.20 and Corollary 25.67}.
Note that $\sigma_{0}(S) \neq 0$;
see~\cite{Ponge2008}*{Section 5} for example.
Taking the complex conjugate gives
the partial inverse $N_{-n} \in \Hpsido^{-2}(M)$ of $L_{-n} = \overline{\Box}_{b}$
and the orthogonal projection $\ovS \in \Hpsido^{0}(M)$ to $\Ker \overline{\Box}_{b}$.

\begin{lemma}
\label{lem:commutator-of-Szego-projection}
	For any $\mu \in \bbR$,
	one has $\comm{L_{\mu}}{S} \in \Hpsido^{1}(M)$.
\end{lemma}

\begin{proof}
	We have
	\begin{equation}
		\comm{L_{\mu}}{S}
		= \comm{L_{n}}{S} + \frac{\sqrt{-1}}{2}(\mu - n)\comm{T}{S}
		= \frac{\sqrt{-1}}{2}(\mu - n)\comm{T}{S}
		\in \Hpsido^{1}(M)
	\end{equation}
	by \cref{lem:commutator-for-Reeb-vector-field}.
\end{proof}

On the other hand,
Hsiao~\cite{Hsiao2010}*{Chapter 7} has studied the distribution kernel of the \Szego projection.
A similar discussion to \cite{Hsiao2015}*{Lemma 4.2} yields

\begin{lemma}
\label{lem:composition-of-Szego-projection-and-its-conjugate}
	The operators $\ovS S$ and $S \ovS$ are smoothing operators.
\end{lemma}

The critical CR GJMS operator $\GJMS$ on $(M, T^{1, 0} M, \theta)$
coincides with
\begin{equation}
	L_{n} L_{n - 2} \dotsm L_{- n + 2} L_{- n}
\end{equation}
modulo $\Hpsido^{2 n + 1}(M)$;
see~\cite{Ponge2008-Book}*{Proposition 3.5.7}.
Set
\begin{equation}
	G_{0}
	\coloneqq N_{- n} N_{- n + 2} \dotsm N_{n - 2} N_{n}
	\in \Hpsido^{- 2 n - 2}(M),
	\qquad
	\Pi_{0}
	\coloneqq S + \ovS
	\in \Hpsido^{0}(M)
\end{equation}
Then modulo $\Hpsido^{- 1}(M)$,
\begin{align}
	G_{0} \GJMS
	&\equiv N_{- n} N_{- n + 2} \dotsm N_{n - 2} N_{n} L_{n} L_{n - 2} \dotsm L_{- n + 2} L_{- n} \\
	&= N_{- n} N_{- n + 2} \dotsm N_{n - 2} (I - S) L_{n - 2} \dotsm L_{- n + 2} L_{- n} \\
	&\equiv N_{- n} N_{- n + 2} \dotsm N_{n - 2} L_{n - 2} \dotsm L_{- n + 2} L_{- n} (I - S) 
		\qquad \text{($\because$ \cref{lem:commutator-of-Szego-projection})}\\
	&\equiv (I - \ovS) (I - S) \\
	&= I - S - \ovS + \ovS S \\
	&\equiv I - \Pi_{0}. 
\end{align}
Thus we have
\begin{equation}
	R_{0}
	\coloneqq G_{0} \GJMS + \Pi_{0} - I \in \Hpsido^{- 1}(M).
\end{equation}
This $G_{0}$ gives an approximation of the partial inverse of $\GJMS$.

\begin{proposition}
	There exists $G_{\infty} \in \Hpsido^{- 2 n - 2}(M)$ such that
	\begin{equation}
		G_{\infty} \GJMS + \Pi_{0} - I
		\in \Hpsido^{- \infty}(M).
	\end{equation}
\end{proposition}

\begin{proof}
	Since the critical GJMS operator annihilates CR pluriharmonic functions,
	we have $\GJMS \Pi_{0} = 0$.
	Hence
	\begin{equation}
		(\Pi_{0})^{2}
		= (G_{0} \GJMS + \Pi_{0}) \Pi_{0}
		= \Pi_{0} + R_{0} \Pi_{0}.
	\end{equation}
	On the other hand,
	$(\Pi_{0})^{2}$ is equal to $\Pi_{0}$ modulo a smoothing operator
	by \cref{lem:composition-of-Szego-projection-and-its-conjugate}.
	Thus we have $R_{0} \Pi_{0} \in \Hpsido^{- \infty}(M)$.
	Take $G_{\infty} \in \Hpsido^{- 2 n - 2}(M)$ such that
	\begin{equation}
		G_{\infty} - \sum_{l = 0}^{k} (- R_{0})^{l} G_{0}
		\in \Hpsido^{- 2 n - k - 3}(M)
	\end{equation}
	for any $k \in \bbN$.
	Then modulo $\Hpsido^{- k - 1}(M)$,
	\begin{align}
		G_{\infty} \GJMS + \Pi_{0}
		&\equiv \sum_{l = 0}^{k} (- R_{0})^{l} G_{0} \GJMS + \Pi_{0} \\
		&= \sum_{l = 0}^{k} (- R_{0})^{l} (I - \Pi_{0} + R_{0}) + \Pi_{0} \\
		&\equiv \sum_{l = 0}^{k} (- R_{0})^{l} (I + R_{0}) \\
		&\equiv I.
	\end{align}
	Therefore $G_{\infty} \GJMS + \Pi_{0} - I$ is a smoothing operator.
\end{proof}

Consider $\GJMS$ as an unbounded closed operator on $L^{2}(M)$ with domain
\begin{equation}
	\Dom \GJMS
	= \Set{u \in L^{2}(M) | \text{$\GJMS u$ in the weak sense is in $L^{2}(M)$}}.
\end{equation}
This domain contains $\HSobolev^{2 n + 2}(M)$ by \cref{prop:mapping-properties-of-Hpsido}.
Conversely,
any $u \in \Dom \GJMS$
is an element of $\HSobolev^{2 n + 2}(M)$ modulo $\Ker \GJMS$ by the lemma below.

\begin{lemma}
\label{lem:domain-of-critical-GJMS-operator}
	For $u \in \Dom \GJMS$,
	one has $u - \Pi_{0} u \in \HSobolev^{2 n + 2}(M)$.
	In particular,
	$\Dom \GJMS = \Ker \GJMS + \HSobolev^{2 n + 2}(M)$.
\end{lemma}

\begin{proof}
	Set
	\begin{equation} \label{eq:left-parametrix}
		R_{\infty}
		\coloneqq G_{\infty} \GJMS + \Pi_{0} - I
		\in \Hpsido^{- \infty}(M).
	\end{equation}
	If $v = \GJMS u \in L^{2}(M)$,
	then
	\begin{equation}
		u - \Pi_{0} u
		= G_{\infty} v - R_{\infty} u
		\in \HSobolev^{2 n + 2}(M).
	\end{equation}
	In particular,
	$u \in \Ker \GJMS + \HSobolev^{2 n + 2}(M)$
	since $\Pi_{0} u \in \Ker \GJMS$.
\end{proof}

\begin{lemma}
\label{lem:GJMS-orthogonal-to-Pi}
	The range $\Ran \GJMS$ of $\GJMS$ is orthogonal to $\Ran \Pi_{0}$ in $L^{2}(M)$.
\end{lemma}

\begin{proof}
	Assume that $u \in \Dom \GJMS$ and $v \in L^{2}(M)$.
	Take a sequence $(v_{j}) \in C^{\infty}(M)$
	such that $v_{j}$ converges to $v$ in $L^{2}(M)$ as $j \to + \infty$.
	Since $\Pi_{0} \in \Hpsido^{0}(M)$,
	the function $\Pi_{0} v_{j}$ is smooth
	and converges to $\Pi_{0} v$ in $L^{2}(M)$ as $j \to + \infty$ also.
	Hence
	\begin{equation}
		\iproduct{\GJMS u}{\Pi_{0} v}_{0}
		= \lim_{j \to \infty} \iproduct{\GJMS u}{\Pi_{0} v_{j}}_{0}
		= \lim_{j \to \infty} \iproduct{u}{\GJMS \Pi_{0} v_{j}}_{0}
		= 0,
	\end{equation}
	which completes the proof.
\end{proof}

\begin{proof}[Proof of \cref{thm:self-adjointness-of-critical-CR-GJMS-operator}]
	We first prove that $\GJMS$ is self-adjoint.
	To this end,
	it is enough to show that $\GJMS$ is symmetric.
	Let $u, v \in \Dom \GJMS$.
	It follows from \cref{lem:domain-of-critical-GJMS-operator} that
	$v^{\prime} \coloneqq v - \Pi_{0} v$ is in $\HSobolev^{2 n + 2}(M)$.
	Take a sequence $(v_{j})$ in $C^{\infty}(M)$
	such that $v_{j}$ converges to $v^{\prime}$ in $\HSobolev^{2 n + 2}(M)$ as $j \to + \infty$.
	Then $\GJMS v_{j}$ converges to $\GJMS v^{\prime} = \GJMS v$ in $L^{2}(M)$ as $j \to + \infty$
	by the continuity of $\GJMS \colon \HSobolev^{2 n + 2}(M) \to L^{2}(M)$.
	We derive from \cref{lem:GJMS-orthogonal-to-Pi} that
	\begin{equation}
		\iproduct{\GJMS u}{v}_{0}
		= \iproduct{\GJMS u}{v^{\prime}}_{0} + \iproduct{\GJMS u}{\Pi_{0} v}_{0}
		= \lim_{j \to \infty} \iproduct{\GJMS u}{v_{j}}_{0}
		= \lim_{j \to \infty} \iproduct{u}{\GJMS v_{j}}_{0}
		= \iproduct{u}{\GJMS v}_{0},
	\end{equation}
	which means that $\GJMS$ is symmetric.

	We next prove that $\GJMS \colon \Dom \GJMS \to L^{2}(M)$ has closed range.
	It suffices to show that
	there exists $\varepsilon > 0$ such that
	\begin{equation}
		\norm{\GJMS u}_{0}
		\geq \varepsilon \norm{u}_{0}
	\end{equation}
	for any $u \in \Dom \GJMS \cap (\Ker \GJMS)^{\perp}$.
	Note that $(\Ker \GJMS)^{\perp} \subset \Ker \Pi_{0}$
	since $\Ran \Pi_{0} \subset \Ker \GJMS$.
	Suppose to the contrary that
	we can take a sequence $(u_{j})$ in $\Dom \GJMS \cap (\Ker \GJMS)^{\perp}$
	such that
	\begin{equation}
		\norm{u_{j}}_{0}
		= 1,
		\qquad
		\norm{\GJMS u_{j}}_{0}
		\leq \frac{1}{j}.
	\end{equation}
	Let $R_{\infty}$ be as in \cref{eq:left-parametrix}.
	Then
	\begin{equation}
		u_{j}
		= G_{\infty} (\GJMS u_{j}) - R_{\infty} u_{j}
	\end{equation}
	is uniformly bounded in $\HSobolev^{2n+2}(M)$.
	By \cref{lem:Rellich's-lemma},
	we may assume that $u_{j}$ converges to some $u \in L^{2}(M)$ as $j \to + \infty$.
	We derive from the definition of $u_{j}$ that
	$u$ is in $(\Ker \GJMS)^{\perp}$ and $\norm{u}_{0} = 1$.
	However,
	since $\norm{\GJMS u_{j}}_{0} \leq 1 / j$,
	we have $u \in \Dom \GJMS$ and $\GJMS u = 0$.
	This is a contradiction.
\end{proof}

Since $\GJMS$ is a range-closed operator,
there exist the partial inverse $G$ of $\GJMS$
and the orthogonal projection $\Pi$ to $\Ker \GJMS$.
Next,
we show that these operators are Heisenberg pseudodifferential operators.

\begin{theorem}
\label{thm:partial-inverse-of-critical-CR-GJMS-operator}
	The operators $G$ and $\Pi$
	are Heisenberg pseudodifferential operators of order $- 2 n - 2$ and $0$ respectively.
	Moreover,
	$\Pi$ coincides with $\Pi_{0}$ modulo $\Hpsido^{- \infty}(M)$.
\end{theorem}

\begin{proof}
	First note that
	\begin{equation}
		\Pi \Pi_{0}
		= \Pi_{0} \Pi
		= \Pi_{0}.
	\end{equation}
	since $\Ran \Pi_{0} \subset \Ker \GJMS$.
	Composing $\Pi$ to \cref{eq:left-parametrix} from the right
	and taking its adjoint,
	we have
	\begin{equation}
		\Pi_{0}
		= \Pi + R_{\infty} \Pi,
		\qquad
		\Pi_{0}
		= \Pi + \Pi (R_{\infty})^{\ast}.
	\end{equation}
	Hence
	\begin{equation}
		\Pi - \Pi_{0}
		= - \Pi (R_{\infty})^{\ast}
		= R_{\infty} \Pi (R_{\infty})^{\ast} - \Pi_{0} (R_{\infty})^{\ast},
	\end{equation}
	which is a smoothing operator.
	In particular,
	$\Pi$ is a Heisenberg pseudodifferential operator of order $0$
	and coincides with $\Pi_{0}$ modulo a smoothing operator.

	Next consider $G$.
	Composing $G$ to \cref{eq:left-parametrix} from the right
	and taking its adjoint give that
	\begin{equation}
		G_{\infty} (I - \Pi)
		= G + R_{\infty} G,
		\qquad
		(I - \Pi) (G_{\infty})^{\ast}
		= G + G (R_{\infty})^{\ast}.
	\end{equation}
	Hence
	\begin{equation}
		G - G_{\infty} (I - \Pi)
		= - R_{\infty} G
		= - R_{\infty}( I - \Pi) (G_{\infty})^{\ast}
			+ R_{\infty} G (R_{\infty})^{\ast},
	\end{equation}
	which is a smoothing operator.
	Therefore $G$ is a Heisenberg pseudodifferential operator of order $- 2 n - 2$.
\end{proof}

This theorem proves \cref{thm:spectrum-of-critical-CR-GJMS-operator}.

\begin{proof}[Proof of \cref{thm:spectrum-of-critical-CR-GJMS-operator}]
	From \cref{prop:mapping-properties-of-Hpsido,thm:partial-inverse-of-critical-CR-GJMS-operator},
	we derive that the partial inverse $G \colon L^{2}(M) \to L^{2}(M)$ is a compact self-adjoint operator.
	Hence the spectrum $\sigma(G)$ of $G$ is bounded
	and consists only of eigenvalues,
	and $0$ is the only accumulation point of $\sigma(G)$.
	Moreover,
	for any non-zero eigenvalue $\lambda$,
	the eigenspace $H_{\lambda} \coloneqq \Ker (G - \lambda)$
	is finite-dimensional,
	and there exists the following orthogonal decomposition:
	\begin{equation}
		L^{2}(M) = \Ker G \oplus \bigoplus_{\lambda \in \sigma(G) \setminus \{0\}} H_{\lambda}.
	\end{equation}
	Furthermore,
	since $G$ maps $\HSobolev^{s}(M)$ to $\HSobolev^{s+2n+2}(M)$,
	the eigenspace $H_{\lambda}$ is a linear subspace of $C^{\infty}(M)$.
	By the definition of the partial inverse,
	$H_{\lambda}$ is the eigenspace of $\GJMS$ with eigenvalue $1 / \lambda$,
	and $\Ker G = \Ker \GJMS$.
	Hence the spectrum $\sigma(\GJMS)$ is discrete and
	consists only of eigenvalues,
	and the eigenspace corresponding to each non-zero eigenvalue
	is a finite-dimensional subspace of $C^{\infty}(M)$.
	Moreover,
	$\Ker \GJMS \cap C^{\infty}(M)$ is dense in $\Ker \GJMS$
	since the orthogonal projection $\Pi$ to $\Ker \GJMS$
	is a Heisenberg pseudodifferential operator of order $0$.
\end{proof}

An argument similar to the proof of \cref{thm:partial-inverse-of-critical-CR-GJMS-operator}
also gives \cref{prop:supplementary-space}.

\begin{proof}[Proof of \cref{prop:supplementary-space}]
	Let $\pi$ be the orthogonal projection to $\overline{\scrP}$.
	Note that $\Pi - \pi$ is the orthogonal projection to $\scrW$.
	Hence it is enough to prove that $\Pi - \pi$ is a smoothing operator.
	Since $\Pi \sim \Pi_{0}$,
	it suffices to show that $\pi - \Pi_{0}$ is a smoothing operator.
	Since $\Ran \Pi_{0} \subset \Ran \pi$,
	\begin{equation}
		\pi \Pi_{0}
		= \Pi_{0} \pi
		= \Pi_{0}.
	\end{equation}
	It follows from \cref{eq:left-parametrix} that
	\begin{equation}
		\Pi_{0}
		= \pi + R_{\infty} \pi,
		\qquad
		\Pi_{0}
		= \pi + \pi (R_{\infty})^{\ast}.
	\end{equation}
	Therefore we have
	\begin{equation}
		\pi - \Pi_{0}
		= - \pi (R_{\infty})^{\ast}
		= R_{\infty} \pi (R_{\infty})^{\ast}
			- \Pi_{0} (R_{\infty})^{\ast},
	\end{equation}
	which is a smoothing operator.
\end{proof}

As an application of results in this section,
we give a necessary and sufficient condition for the zero CR $Q$-curvature problem.

\begin{proof}[Proof of \cref{prop:zero-CR-Q-curvature-problem}]
	As we saw in the introduction,
	$Q \perp (\Ker \GJMS \cap C^{\infty}(M))$
	if there exists a contact form with zero $Q$-curvature.
	Conversely,
	assume that $Q$ is orthogonal to $\Ker \GJMS \cap C^{\infty}(M)$.
	It follows from \cref{thm:spectrum-of-critical-CR-GJMS-operator} that
	$Q$ is in fact orthogonal to $\Ker \GJMS$.
	Then $\Upsilon \coloneqq - G Q \in C^{\infty}(M)$ and $P \Upsilon = - Q$.
	Hence $\whxth \coloneqq e^{\Upsilon} \theta$ satisfies $\widehat{Q} = 0$.
\end{proof}

\section*{Acknowledgements}
The author is grateful to Charles Fefferman, Kengo Hirachi, and Paul Yang for helpful comments.
A part of this work was carried out during his visit to Princeton University
with the support from The University of Tokyo/Princeton University
Strategic Partnership Teaching and Research Collaboration Grant,
and the Program for Leading Graduate Schools, MEXT, Japan.
He would like to thank Princeton University for its kind hospitality.
The author is also grateful to the referee
for a careful reading and some valuable suggestions,
which improves this manuscript.

\bibliography{my-reference,my-reference-preprint}

\end{document}